\documentclass[11pt]{amsart}

\usepackage{amsmath}
\usepackage{amssymb}
\usepackage{graphicx}
 \usepackage{commath,mathtools,setspace}
 \usepackage{paralist}
 \usepackage{color}
\usepackage{comment}
\usepackage[colorlinks=true,linkcolor=blue,citecolor=blue]{hyperref}

\definecolor{red}{RGB}{255,0,0}
\definecolor{green}{RGB}{0,100,0}
\definecolor{blue}{RGB}{0,0,255}

\newtheorem{theorem}{Theorem}[section]
\newtheorem{thmx}{Theorem}
\newtheorem{corollary}[theorem]{Corollary}
\newtheorem{lemma}[theorem]{Lemma}

\theoremstyle{definition}
\newtheorem{definition}[theorem]{Definition}

\newtheorem{remark}[theorem]{Remark}

\numberwithin{equation}{section}

\newcommand{\sign}{\mathop{\rm sign}}

\renewcommand{\P}{\mathbb{P}}

\newcommand{\R}{\mathbb{R}}

\newcommand{\N}{\mathbb{N}}

\title[Zeros of Angelesco-Jacobi polynomials]{Interlacing and monotonicity of zeros of Angelesco-Jacobi polynomials} 
\author[A. Mart\'{\i}nez-Finkelshtein]{Andrei Mart\'{\i}nez-Finkelshtein}

\address[AMF]{Department of Mathematics, Baylor University, TX, USA, and Department of Mathematics, University of Almer\'{\i}a, Spain}

\email{A\_Martinez-Finkelshtein@baylor.edu}

\author[R.~Morales]{Rafael Morales}

\address[RM]{Department of Mathematics, Baylor University, TX, USA}

\email{rafael\_morales2@baylor.edu}

\date{\today}



\begin{document}

\begin{abstract}
Information about the behavior of zeros of classical families of multiple or Hermite--Pad\'e orthogonal polynomials as functions of the intrinsic parameters of the family is scarce. We establish the interlacing properties of the zeros of Angelesco-Jacobi polynomials when one of the three main parameters is increased by $1$, extending the work of \cite{dos2017monotonicity}. We also show their monotonicity with respect to (large values) of the parameter $\gamma$, representing in the electrostatic model of the zeros the size of the positive charge fixed at the origin, as well as monotonicity with respect to the endpoint of the interval of orthogonality. These results are extended to zeros of multiple Jacobi-Laguerre and Laguerre-Hermite polynomials using asymptotic relations between these families. 
\end{abstract}

\maketitle


\section{Introduction}

The Jacobi polynomials $P_n^{(\alpha,\beta)}$ can be defined using the  Rodrigues formula, 
\begin{equation}
	\label{Rodrigues}
	P_n^{(\alpha,\beta)} (z)=\frac{1}{2^n n!} (z-1)^{-\alpha}
	(z+1)^{-\beta} \left( \frac{d}{dz} \right)^n \left[
	(z-1)^{n+\alpha} (z+1)^{n+\beta}\right].
\end{equation}
For $\alpha, \beta>-1$, they are orthogonal on $[a,b]=[-1,1]$ with respect to the weight 
$$
w_{\alpha,\beta}(x)=(1-x)^\alpha (1+x)^\beta, 
$$ 
that is,
$$
\int_{-1}^1 P_n^{(\alpha,\beta)}(x) x^k w_{\alpha,\beta}(x)\, dx=0, \quad k=0, 1, \dots, n-1.
$$
In consequence, all zeros 
\begin{equation}
    \label{zerosJacobi}
    -1< x_{n,1}(\alpha,\beta)< \dots < x_{n,n}(\alpha,\beta)<  1
\end{equation}
of $P_n^{(\alpha,\beta)}(x)$ are real, simple, and belong to $(-1,1)$. Their electrostatic interpretation, found by Stieltjes in 1885, is an elegant result that offers intuition about their behavior as functions of the parameters $\alpha, \beta$. Namely, zeros \eqref{zerosJacobi} provide a unique minimizer of the logarithmic energy
 \begin{equation*}
    E_{total}(X)=E_{mutual}(X)+\sum_{x\in X}\varphi(x), \quad X=\{x_1,\dots,x_n\} \subset [-1,1]^n,
\end{equation*}
among all $n$ point configurations on $[-1,1]$, where
\begin{equation*}
    E_{mutual}(X):=\sum_{1\leq k<j\leq n}\log\frac{1}{\lvert x_k - x_j\rvert},
\end{equation*}
and the external field $\varphi(x)$ acting on $X$ is given by two fixed charges of mass $(\beta+1)/2$ and $(\alpha+1)/2$ at $-1$ and $1$, respectively:
\begin{equation*}
    \varphi(x)=\frac{\alpha+1}{2}\log\frac{1}{\lvert x-1\rvert} + \frac{\beta+1}{2}\log\frac{1}{\lvert x+1\rvert}.
\end{equation*}
Since we conclude that the size of $\alpha$ (resp., $\beta$) is proportional to the strength of repulsion of the endpoint $1$ (resp., $-1$), it is a basis for the conjecture that each zero $x_{n,j}(\alpha,\beta)$ should behave monotonically with respect to $\alpha$ and $\beta$. This intuition was rigorously substantiated by Szeg\H{o} in \cite[p.~115]{szego1975orthogonal}: 
 \begin{thmx}  \label{thm:Szego}
Given $n\in\mathbb{N}$, the zeros \eqref{zerosJacobi} of $P_n^{(\alpha,\beta)}(x)$ decrease with respect to $\alpha$ and increase with respect to $\beta$. 
\end{thmx}

The monotonicity with respect to the parameters can also be recast in terms of the interlacing of zeros of polynomials from the same family. In what follows, we denote by $\P_n(\R)$ the family of real--rooted algebraic polynomials of 
degree $n$.
\begin{definition}[Interlacing]\label{definition;interlacing} Let $(a,b)\subset \R$, $p$ and $q$ two polynomials, and let $ a<\lambda_1(p) < \cdots < \lambda_k(p)<b$ and $a<\lambda_1(q)< \cdots< \lambda_{j}(q)<b$ be the zeros of $p$ and $q$, respectively, all simple, that that belong to $(a,b)$. We say that $q$ \textbf{interlaces $p$ on $(a,b)$}, and denote it by $p \prec q$ on $ (a,b)$, if  $j=k$  and 
\begin{equation}\label{SamedegreeInter} 
    \lambda_1(p) < \lambda_1(q) < \lambda_2(p) < \lambda_2(q) <\cdots < \lambda_k(p) <\lambda_k(q),
\end{equation}
or if $j=k-1$  and 
\begin{equation}\label{LessdegreeInter}
    \lambda_1(p) < \lambda_1(q) < \lambda_2(p) < \lambda_2(q) < \cdots <  \lambda_{n-1}(p) < \lambda_{k-1}(q) < \lambda_k(p).
\end{equation}

Moreover, if $(a,b)=\R$, we simplify the notation by simply writing $ p \prec q$.
\end{definition}

It is clear that the monotonicity of zeros of Jacobi polynomials (Theorem \ref{thm:Szego}) is equivalent to the interlacing of families like $P_n^{(\alpha,\beta)}$ and $P_n^{(\alpha+t,\beta)}$ for some range of $t>0$ (and similarly, for the parameter $\beta$). A result is this direction was found in \cite{driver2008interlacing}:
\begin{thmx} \label{thm:zerointerlacingJacobi}
     Let  $\alpha,\beta>-1$ and $t\in (0,2)$. Then 
     \begin{equation*}
         P_n^{(\alpha+t,\beta)}(x)\prec P_n^{(\alpha,\beta)}(x), \qquad  P_n^{(\alpha,\beta)}(x)\prec P_n^{(\alpha,\beta+t)}(x). 
     \end{equation*}
\end{thmx}

The Laguerre polynomials $L_n^{(\gamma)}$ can be defined by their explicit representation in terms of the terminating generalized hypergeometric series,
$$
L_n^{(\gamma)}(x)=\frac{(\gamma+1)_n}{n !}{ }_1 F_1(-n ; \gamma+1 ; x),
$$
and for $\gamma >-1$ satisfy the orthogonality relations
$$
 \int_0^{\infty} x^j L_n^{(\gamma)}(x)  x^\gamma e^{-x} \, d x \\
=0, \quad j=0, 1, \dots, n-1.
$$
The Hermite polynomials $H_n$ are the third classical family of polynomials; they satisfy the orthogonality relations
$$
\int_{\mathbb{R}} x^j  H_n(x) e^{-x^2} d x=0,  \quad j=0, 1, \dots, n-1,
$$
see \cite[\S 4.6]{MR2542683}. There are well-known limit relations that connect all three classical families; see, e.g., \cite[Section 18.7(iii)]{NIST:DLMF}.

\medskip

\textit{Multiple} (or \textit{Hermite--Pad\'e}) \textit{polynomials} are a natural generalization of the classical notion of orthogonal polynomials. Typically, they come in two flavors, but this paper focuses on the so-called Type II polynomials. To define them, we need a multi-index $\Vec{n}=\left( n_1,\dots,n_r\right)\in \mathbb N^r$ and a $r$-tuple of positive Borel measures $\left\{ \mu_j\right\}_{j=1}^r$, with an additional assumption that all the moments 
$$
\int |x|^k \, d\mu_j(x), \quad j=1,\dots, r, \quad k\ge 0,
$$
exist. A monic polynomial $P_{\Vec{n}}(x)$ of degree $\left|\Vec{n} \right |:= n_1+\cdots+n_r$, if it exists, is called a \textit{Type II multiple orthogonal polynomial} if it satisfies the following orthogonality conditions:
\begin{equation*}
    \int x^k P_{\Vec{n}}(x)d\mu_j(x)=0 \quad \text{for} \quad k=0,1,\dots,n_j-1 \quad \text{and} \quad j=1,\dots,r .
\end{equation*}
In general, the existence of such a polynomial for a specific $\vec n \in \N^r$ cannot be guaranteed. If they do exist for all multi-indices $\Vec{n}\in \N^r$, we say that the system of measures $\left\{ \mu_j\right\}_{j=1}^r$ is \textit{perfect}. One of the simplest examples of a perfect system is the so-called \textit{Angelesco case}, when the convex hull of the support of each $\mu_j$ is a closed interval $I_j\subset \mathbb R$ such that the interiors of $\left\{ I_j\right\}_{j=1}^r$ are pairwise disjoint. If measures $\mu_j$ are additionally classical (or more generally, semiclassical) weights of orthogonality, the corresponding systems receive the name of multiple Jacobi-Jacobi (or Angelesco-Jacobi), Jacobi-Laguerre and Laguerre-Hermite orthogonal polynomials; see \cite{MR2542683} and the definitions below for details. As their classical counterparts, they may also depend on a number of parameters. 

One of the features of the Angelesco system is that all zeros in $P_{\Vec{n}}(x)$ are real and simple, and each interval $I_j$ contains exactly $n_j$ of these zeros. However, their electrostatic model, in the spirit of Stieltjes' work, was given only recently, in \cite{Martinez-Finkelshtein:2022aa}. Unlike the standard case, the model involves charges of opposite sign, some of them located at the zeros of the so-called electrostatic partner $S$ of $P_{\Vec{n}} $, a polynomial that can be constructed using an explicit formula. In the Angelesco case, the zeros of $S$ partially interlace with those of $P_{\Vec{n}} $ (and the location of the zeros corresponds to a critical point of the gradient of the vector energy rather than a global minimum). The complexity of this vector equilibrium makes the analysis of the behavior of the zeros of $P_{\Vec{n}} $ as functions of the parameters of the weight highly nontrivial. 

In this paper, we will present results about the interlacing and monotonicity of the zeros of multiple Angelesco-Jacobi, Jacobi-Laguerre, and Laguerre-Hermite orthogonal polynomials. All main results are collected in Section \ref{sec:main}. Section \ref{sec:Them22} contains some structure relations for Angelesco-Jacobi polynomials, as well as proof of the interlacing properties of their zeros. The behavior of the arithmetic and geometric means of these zeros is proven in Section \ref{sec:mean}, and these results are used in the proof of the monotonicity properties in Sections \ref{sec:Monotonicity} and \ref{sec:monotA}. Finally, the extension of some of these results to the zeros of the Jacobi-Laguerre and Laguerre-Hermite polynomials is established in Section \ref{sec:Laguerre}.

\section{Main results} \label{sec:main}

Angelesco-Jacobi, known also as Jacobi-Jacobi polynomials (see \cite{Aptekarev:97}), are Hermite--Pad\'e polynomials $P_{\vec n}$, $\vec n=(n_1, n_2)$, satisfying orthogonality relations
\begin{equation*}
\begin{split}
    &\int_{a}^{0} x^{k} P_{\vec n}(x)(1-x)^{\alpha}(x-a)^{\beta}|x|^{\gamma}  \, d x=0, \quad k=0,1,2, \dots, n_1-1, \\
&\int_{0}^{1} x^{k} P_{\vec n}(x)(1-x)^{\alpha}(x-a)^{\beta}x^{\gamma} \, d x=0, \quad k=0,1,2, \dots, n_2-1,
\end{split}
\end{equation*}
with $a<0$ and $\alpha,\beta,\gamma>-1$. To make the dependence on the parameters explicit, we denote these polynomials by $P_{\vec n}^{(\alpha,\beta,\gamma)}(x;a)$.\footnote{\, Notice that the role of the parameters $\alpha, \beta, \gamma$ can differ in the literature.} Thus, this is a type II multiple orthogonal polynomial for the Angelesco system given by the same weight function $\omega (x)=(1-x)^\alpha(x-a)^\beta\abs{x}^\gamma$ on the intervals $[a,0]$ and $[0,1]$.

We will consider the diagonal case
\begin{equation}
    \label{diagcase}
    \vec n = (n, n);
\end{equation}
to simplify notation, in what follows and under assumption \eqref{diagcase}, we will write $P_{  n}^{(\alpha,\beta,\gamma)}(x;a)$ instead of $P_{\vec n}^{(\alpha,\beta,\gamma)}(x;a)$, keeping in mind that $$\deg P_{  n}^{(\alpha,\beta,\gamma)}(x;a)=2n.$$
Correspondingly, exactly $n$ zeros of $P_{n}^{(\alpha,\beta,\gamma)}(x;a)$ are in $(a,0)$, and the other $n$, in $(0,1)$.

As pointed out in the Introduction, we assume these polynomials to be monic. They can be defined via a Rodrigues formula \cite[Section 3.5]{MR1808581}
\begin{equation}
    \label{RodriguezForm}
    \begin{split}
        (1-x)^\alpha(x-a)^\beta x^\gamma c_n(\alpha,\beta,\gamma) & P^{(\alpha,\beta,\gamma)}_{n}(x;a)\\
        & =\frac{(-1)^n}{n!}\left( \frac{d}{dx} \right)^n\left[(1-x)^{\alpha+n}(x-a)^{\beta+n}x^{\gamma+n}\right]
    \end{split}
\end{equation}
with
\begin{equation}
    c_n(\alpha,\beta, \gamma):=\binom{3n+\alpha+\beta+\gamma}{n};
\end{equation}
here and in what follows, we use the generalized definition of the binomial coefficients,
$$
\binom{a}{k} := \frac{\Gamma(a+1)}{\Gamma(a-k+1) k! }, \quad k\in \mathbb N \cup \{0\}, \quad  a> k.
$$
Among the consequences of \eqref{RodriguezForm} are the raising operator identity
\begin{equation}\label{raisingJA}
    \begin{split}
        (2n+\alpha+&\beta+\gamma+1) (1-x)^\alpha  (x-a)^\beta x^\gamma   P^{(\alpha,\beta,\gamma)}_{n}(x;a)\\
        & =- \frac{d}{dx} \left[(1-x)^{\alpha+1}(x-a)^{\beta+1}x^{\gamma+1} P^{(\alpha+1,\beta+1,\gamma+1)}_{n-1}(x;a)\right],
    \end{split}
\end{equation}
and an explicit expression (see \cite[Section 23.3.1]{MR2542683}),
\begin{equation}
\label{explicitJA}
    c_n(\alpha,\beta,\gamma) P_{n}^{(\alpha,\beta,\gamma)}(x;a)=\sum_{k=0}^n\sum_{j=0}^{n-k}d^{(n)}_{j,k}( \alpha,\beta,\gamma) (x-1)^{n-k}x^{k+j}(x-a)^{n-j},
\end{equation}
where for $0\le j \le n-k$ and $  0\le k \le n$,
\begin{equation}
\label{explicitCoefficientsJA}
  d^{(n)}_{j,k}( \alpha,\beta,\gamma):=  \binom{n+\alpha}{k}\binom{n+\beta}{j}\binom{n+\gamma}{n-k-j}.
\end{equation}

The zero interlacing of consecutive $ P_n^{(\alpha,\beta,\gamma)}(x;a)$, in the spirit of Theorem \ref{thm:zerointerlacingJacobi}, is not known (for the interlacing of zeros of the nearest neighbor polynomials, see \cite{MR2876509}). However, de Santos \cite{dos2017monotonicity} proved the following result (see \cite[Lemma 2.1, iv)]{dos2017monotonicity}), that we formulate using Definition \ref{definition;interlacing}: 
\begin{thmx}  \label{dossantosformula}
Let $\alpha, \beta, \gamma >-1$, $a<0$, and $n\in\N$, $n\ge 2$. Then 
 \begin{align*}
     P_n^{(\alpha,\beta,\gamma)}(x;a)  &\prec  P_{n-1}^{(\alpha+1,\beta+1,\gamma+1)}(x;a) \text{ on } (a,0),   \\
     P_n^{(\alpha,\beta,\gamma)}(x;a)  &\prec  P_{n-1}^{(\alpha+1,\beta+1,\gamma+1)}(x;a) \text{ on } (0,1).
  \end{align*}
\end{thmx}
It is convenient to remind the reader that $P_n$ (resp., $P_{n-1}$) has exactly $n$ (resp., $n-1$) zeros on $(a,0)$ and on $(0,1)$. 
\begin{remark}
    \label{remark1}
    A simple observation is that the assertion of this theorem is equivalent to 
    $$
    x (x-a) P_{n-1}^{(\alpha+1,\beta+1,\gamma+1)}(x;a) \prec P_n^{(\alpha,\beta,\gamma)}(x;a) .
    $$
\end{remark}
This result shows that the zeros of two consecutive diagonal Angelesco-Jacobi polynomials interlace at each subinterval if the decrease in degree is accompanied by an increase in the value of each parameter $\alpha, \beta, \gamma$ in $1$, see Figure \ref{fig:JAdeg5}.
\begin{figure}
    \centering
    \includegraphics[width=0.75\linewidth]{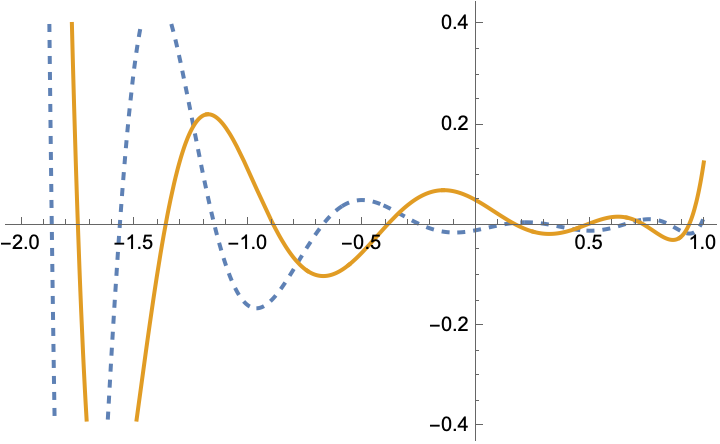}
    \caption{Graph of the polynomials $P_5^{(-1/2,1,0)}(x;-2)$ (dashed line) and $P_4^{(1/2,2,1)}(x;-2)$ (thick line) on $[-2,1]$.}
    \label{fig:JAdeg5}
\end{figure}

One of the main results of this paper is the following refinement of Theorem \ref{dossantosformula}:
\begin{theorem}\label{entrelazamiento}
Let $\alpha, \beta, \gamma >-1$,  $a<0$, and $n\in \N$, $n\ge 2$. Then 
\begin{enumerate}[(i)]
    \item \label{thm21i}
    $  x(x-a)P_{n-1}^{(\alpha+1,\beta+1,\gamma+1)}(x;a)\prec P_n^{(\alpha+1,\beta,\gamma)}(x;a) \prec  P_n^{(\alpha,\beta,\gamma)}(x;a)$; 
    
    \item  \label{thm21ii}
    $ P_n^{(\alpha,\beta,\gamma)}(x;a)\prec P_n^{(\alpha,\beta+1,\gamma)}(x;a) \prec x(x-1)P_{n-1}^{(\alpha+1,\beta+1,\gamma+1)}(x;a)$;  

    \item \label{thm21iii}\begin{align*}
   & P_n^{(\alpha,\beta,\gamma+1)}(x;a)  \prec P_n^{(\alpha,\beta,\gamma)}(x;a) \prec P_{n-1}^{(\alpha+1,\beta+1,\gamma+1)}(x;a) \quad \text{on } (a,0), \\
   & P_n^{(\alpha,\beta,\gamma)}(x;a)  \prec P_n^{(\alpha,\beta,\gamma+1)}(x;a)\prec P_{n-1}^{(\alpha+1,\beta+1,\gamma+1)}(x;a)  \quad \text{on } (0,1).
\end{align*}
\end{enumerate}
\end{theorem}

Although the electrostatic model for $ P_{n}^{(\alpha,\beta,\gamma)}(x; a)$, mentioned above, involves both positive and negative moving charges and is quite complicated to analyze, it still provides a correct intuition on the behavior of zeros as functions of the parameters $\alpha$ and $\beta$ (representing the size of the positive charges fixed at the endpoints $x=a$ and $x=1$, respectively), as a result found in \cite{dos2017monotonicity} shows:
 \begin{thmx} \label{dossantosformula2}
     The zeros of $ P_{n}^{(\alpha,\beta,\gamma)}(x; a)$ decrease with respect to $\alpha$ and increase with respect to $\beta$. Furthermore, if $\alpha=\beta$ and $a=-1$, the negative zeros decrease with respect to $\gamma$, and the positive zeros increase with respect to $\gamma$.
    \end{thmx}

This theorem claims the monotonicity of zeros as functions of the parameter $\gamma$ (the ``size'' of the positive charge at $x=0$) in the most straightforward symmetric setting. When the symmetry is broken, the electrostatic model allows us to expect monotonicity, at least for large values of $\gamma$, when the charge at the origin outpowers the interaction of individual positive and negative unit charges. This is indeed the case, as the following result shows:
\begin{theorem}\label{thm:monotonicity}
Let $\alpha,\beta>-1$,  $a<0$, and $n\in\mathbb{N}$. Then, for all sufficiently large values of $\gamma$, the negative zeros of $P^{(\alpha,\beta,\gamma)}_{n}(x;a)$ decrease with respect to $\gamma$, and its positive zeros increase with respect to $\gamma$.
\end{theorem}
We conjecture that the assertion is valid for all values of the parameter $\gamma>-1$.

Finally, although intuitively obvious, we establish the monotonicity of all zeros of the Angelesco-Jacobi polynomials as functions of the left endpoint $a$:
\begin{theorem} \label{thm:monotonicityA}
    Let $n\in\mathbb{N}$,  $a<0$ and  $\alpha, \beta,\gamma>-1$. The zeros of the polynomial $P_{n}^{(\alpha,\beta+1,\gamma)}(x;a)$ are increasing with respect to $a$.
\end{theorem}
This result, although predictable for the zeros in $(a,0)$, is less obvious for the whole interval in $(a,1)$. 

\medskip

Given $a<0$ and $\beta,\gamma>-1$, the \textit{Jacobi-Laguerre polynomials} $L_{\vec{n}}^{(\beta,\gamma)}(x;a)$, $\vec n=(n_1, n_2)$,  are Hermite--Pad\'e polynomials that satisfy the orthogonality relations
\begin{equation*}
\begin{split}
    &\int_{a}^{0} x^{k} L^{(\beta,\gamma)}_{\vec{n}}(x)(x-a)^{\beta}|x|^{\gamma}e^{-x}  \, d x=0, \quad k=0,1,2, \dots, n_1-1, \\
&\int_{0}^{\infty} x^{k} L^{(\beta,\gamma)}_{\vec{n}}(x)(x-a)^{\beta}x^{\gamma}e^{-x} \, d x=0, \quad k=0,1,2, \dots, n_2-1.
\end{split}
\end{equation*}

The \textit{Laguerre-Hermite polynomials}  $H_{\vec{n}}^{(\gamma)}(x;a)$, for $\gamma>-1$, satisfy 
\begin{equation*}
\begin{split}
    &\int_{-\infty}^{0} x^{k} H^{(\gamma)}_{\vec{n}}(x)|x|^{\gamma}e^{-x^2}  \, d x=0, \quad k=0,1,2, \dots, n_1-1, \\
&\int_{0}^{\infty} x^{k} H^{(\gamma)}_{\vec{n}}(x)x^{\gamma}e^{-x^2} \, d x=0, \quad k=0,1,2, \dots, n_2-1.
\end{split}
\end{equation*}
As in the classical case, these families of multiple orthogonal polynomials are connected by the following asymptotic relations (see \cite[formulas (3.55) and (3.57)]{MR1808581}):
\begin{align}
 \label{LaguerreLimit}
    L_{\vec{n}}^{(\beta,\gamma)}(x;a) & =\lim_{\alpha\to\infty}\alpha^{n_1+n_2}P_{\vec{n}}^{(\alpha,\beta,\gamma)}(x/\alpha;a/\alpha),
\\
       \label{HermiteLimit}
    H_{\vec{n}}^{(\gamma)}(x) & =\lim_{\alpha\to\infty} \alpha^{(n_1+n_2)/2}P_{\vec{n}}^{(\alpha,\alpha,\gamma)}(x/\sqrt{\alpha};-1).
\end{align}

As in the case of Angelesco-Jacobi polynomials, we will focus on the diagonal \eqref{diagcase}, adopting the notation $L_n^{(\alpha,\beta)}(x;a)$ and $ H_n^{(\beta)}(x)$ instead of $ L_{\vec{n}}^{(\alpha,\beta )}(x;a)$ and $ H_{\vec{n}}^{(\beta)}(x)$, respectively. Again, being in each case an Angelesco system, exactly $n$ zeros are strictly positive, while the other $n$ are strictly negative (and in the interval $(a,0)$ in the case of polynomials $L_n^{(\alpha,\beta)}(x;a)$).

Using the asymptotic relations \eqref{LaguerreLimit}--\eqref{HermiteLimit}, we extend the interlacing result of Theorem \ref{entrelazamiento} to the Jacobi-Laguerre and Laguerre-Hermite polynomials:
\begin{theorem} \label{thm:interlacingLaguerre}
    Let $a<0$ and $\beta,\gamma>-1$, the following interlacing holds:
    \begin{enumerate}
    \item 
        \begin{equation*}
    L_n^{(\beta,\gamma)}(x;a)\prec  L_n^{(\beta+1,\gamma)}(x;a) \prec  xL_{n-1}^{(\beta+1,\gamma+1)}(x;a)    .
        \end{equation*}    
     \item 
        \begin{equation*}
            \begin{split}
      L_n^{(\beta,\gamma+1)}(x;a) \prec L_n^{(\beta,\gamma)}(x;a) \prec L_{n-1}^{(\beta+1,\gamma+1)}(x;a) \quad & \text{on }   (a,0), \\
   L_n^{(\beta,\gamma)}(x;a) \prec L_n^{(\beta,\gamma+1)}(x;a) \prec L_{n-1}^{(\beta+1,\gamma+1)}(x;a)\quad & \text{on } (0,\infty).
            \end{split}          
        \end{equation*} 
        \item 
        \begin{equation*}
            \begin{split}
    H_n^{(\gamma+1)}(x) \prec H_n^{(\gamma)}(x)\prec H_{n-1}^{(\gamma+1)}(x) \quad & \text{on } (-\infty,0), \\
   H_n^{(\gamma)}(x) \prec H_n^{(\gamma+1)}(x)\prec  H_{n-1}^{(\gamma+1)}(x) \quad & \text{on } (0,\infty).
            \end{split}
        \end{equation*}
    \end{enumerate}
\end{theorem}
Finally, the monotonicity of zeros as a function of the end point $a$ (Theorem~\ref{thm:monotonicityA}) is trivially extended to Jacobi-Laguerre polynomials $L_{ n}^{(\beta,\gamma)}(x;a)$ via the limit relation \eqref{LaguerreLimit}.


\section{Proof of Theorem \ref{entrelazamiento}} \label{sec:Them22}

Theorem \ref{entrelazamiento} is a consequence of the following structure relations for Angelesco-Jacobi polynomials. To our knowledge, they are new and can have an independent interest. To formulate this result, we define for $n\in \N$ and $\alpha, \beta, \gamma>-1$ the coefficients
\begin{equation} \label{defParameters}
\begin{split}
    \mathfrak A_n & =  \mathfrak A_n(\alpha, \beta, \gamma):=\frac{2n+\alpha+\beta+\gamma+1}{3n+\alpha+\beta+\gamma+1}>0, \\ 
    \mathfrak  B_n & =\mathfrak B_n(\alpha, \beta, \gamma) := 1- \mathfrak  A_n =\frac{n}{3n+\alpha+\beta+\gamma+1}>0
\end{split}
\end{equation}
(where we omit the dependence on the rest of the parameters when these values are clear).
\begin{theorem}\label{identity} 
Let $n\in\mathbb{N}$, $\alpha, \beta,\gamma>-1$, and $\mathfrak A_n   =  \mathfrak A_n(\alpha, \beta, \gamma)$ and $\mathfrak B_n  =  \mathfrak B_n(\alpha, \beta, \gamma)$ be as defined above.  Then
\begin{enumerate}[(i)]
    \item  
    $$
     P_n^{(\alpha+1,\beta,\gamma)}(x;a)= \mathfrak A_n P_n^{(\alpha,\beta,\gamma)}(x;a)+\mathfrak B_n x(x-a)P_{n-1}^{(\alpha+1,\beta+1,\gamma+1)}(x;a);
    $$
    \item 
    $$
    P_n^{(\alpha,\beta+1,\gamma)}(x;a)= \mathfrak A_n P_n^{(\alpha,\beta,\gamma)}(x;a)+ \mathfrak B_n (x-1)xP_{n-1}^{(\alpha+1,\beta+1,\gamma+1)}(x;a);
    $$
    \item 
    $$
     P_n^{(\alpha,\beta,\gamma+1)}(x;a)=\mathfrak A_n P_n^{(\alpha,\beta,\gamma)}(x;a)+\mathfrak  B_n (x-a)(x-1)P_{n-1}^{(\alpha+1,\beta+1,\gamma+1)}(x;a).
    $$    
\end{enumerate}
\end{theorem}
\begin{proof}
From the explicit expression \eqref{explicitCoefficientsJA} and the identity
$$
\binom{n}{k}=\binom{n-1}{k-1}+\binom{n-1}{k}
$$
it follows that for $k\geq 1$,
\begin{equation*}
d^{(n)}_{j,k}( \alpha+1,\beta,\gamma)=d^{(n)}_{j,k}( \alpha,\beta,\gamma)+d^{(n-1)}_{j,k-1}( \alpha+1,\beta+1,\gamma+1).
\end{equation*}
Using it in \eqref{explicitJA}, we get that
\begin{align*}
    c_{n}(\alpha+1,\beta,\gamma)P_n^{(\alpha+1,\beta,\gamma)}(x;a)   = & \sum_{k=1}^{n}\sum_{j=0}^{n-k}d^{(n-1)}_{j,k-1}( \alpha+1,\beta+1,\gamma+1)(x-1)^{n-k}x^{k+j}(x-a)^{n-j}\\ & + 
    \sum_{j=0}^{n}d^{(n)}_{j,0}( \alpha+1,\beta,\gamma)(x-1)^{n}x^{j}(x-a)^{n-j} \\
    & +\sum_{k=1}^n\sum_{j=0}^{n-k}d^{(n)}_{j,k}( \alpha,\beta,\gamma)(x-1)^{n-k}x^{k+j}(x-a)^{n-j}.
\end{align*}
Notice that $d^{(n)}_{j,0}( \alpha+1,\beta,\gamma)=d^{(n)}_{j,0}( \alpha,\beta,\gamma)$, so that
\begin{multline*}
    c_{n}(\alpha+1,\beta,\gamma)P_n^{(\alpha+1,\beta,\gamma)}(x;a)=\sum_{k=0}^n\sum_{j=0}^{n-k}d^{(n)}_{j,k}( \alpha,\beta,\gamma)(x-1)^{n-k}x^{k+j}(x-a)^{n-j} \\
    +\sum_{k=1}^{n}\sum_{j=0}^{n-k}d^{(n-1)}_{j,k-1}( \alpha+1,\beta+1,\gamma+1)(x-1)^{n-k}x^{k+j}(x-a)^{n-j}.
\end{multline*}
Shifting the summation variable $k$, we get
\begin{multline*}
    c_{n}(\alpha+1,\beta,\gamma)P_n^{(\alpha+1,\beta,\gamma)}(x;a)=\sum_{k=0}^n\sum_{j=0}^{n-k}d^{(n)}_{j,k}( \alpha,\beta,\gamma)(x-1)^{n-k}x^{k+j}(x-a)^{n-j}\\
    +x(x-a)\sum_{k=0}^{n-1}\sum_{j=0}^{n-1-k}d^{(n-1)}_{j,k}( \alpha+1,\beta+1,\gamma+1)(x-1)^{n-1-k}x^{k+j}(x-a)^{n-1-j}.
\end{multline*}
After dividing the identity by $ c_{n}(\alpha+1,\beta,\gamma)$ and using \eqref{explicitJA} we get \textit{(i)}.

The rest of the assertions are proved using similar arguments.    
\end{proof}

\medskip 

\begin{proof}[Proof of Theorem \ref{entrelazamiento}]
By Remark \ref{remark1}, we have that 
$$
x (x-a) P_{n-1}^{(\alpha+1,\beta+1,\gamma+1)}(x;a) \prec P_n^{(\alpha,\beta,\gamma)}(x;a).
$$
These are polynomials with zeros in $[a,1)$ and degree exactly $2n$. It is a well-known and easy to prove the fact that their convex combination preserves interlacing (see, e.g., \cite{MR1179101, MR1268783} or \cite[Lemma 4.5]{marcus2015interlacing}), which implies that with $
\mathfrak A_n  >0$, $\mathfrak  B_n   >0$ defined in \eqref{defParameters}, we have
\begin{align*}
     x(x-a) & P_{n-1}^{(\alpha+1,\beta+1,\gamma+1)}(x;a)\prec \\ 
     & \mathfrak A_n P_n^{(\alpha,\beta,\gamma)}(x;a)+\mathfrak B_n x(x-a)P_{n-1}^{(\alpha+1,\beta+1,\gamma+1)}(x;a)\prec  P_n^{(\alpha,\beta,\gamma)}(x;a).
 \end{align*}
The identity \textit{(i)} of Theorem \ref{identity} implies part \textit{(i)} of the theorem, that is, 
$$  x(x-a)P_{n-1}^{(\alpha+1,\beta+1,\gamma+1)}(x;a)\prec P_n^{(\alpha+1,\beta,\gamma)}(x;a) \prec  P_n^{(\alpha,\beta,\gamma)}(x;a). 
$$ 
The assertion \textit{ (ii)} of the theorem is proved similarly.

Let us prove \textit{(iii)}, and specifically, the interlacing on $(a,0)$. By Theorem \ref{dossantosformula}, 
 \begin{equation} \label{eq:ordering}
     P_n^{(\alpha,\beta,\gamma)}(x;a)   \prec  P_{n-1}^{(\alpha+1,\beta+1,\gamma+1)}(x;a) \text{ on } (a,0).
 \end{equation}

In what follows, for $\alpha,\beta,\gamma>-1$, we denote by 
\begin{equation}
    \label{def:Zeros}
    x_{n,1}(\alpha,\beta,\gamma;a) < x_{n,2}(\alpha,\beta,\gamma;a) < \dots < x_{n,2n}(\alpha,\beta,\gamma;a) 
\end{equation}
the zeros of $P^{(\alpha,\beta,\gamma)}_{n }(x;a)$. Recall that $a<x_{n,j}(\alpha,\beta,\gamma;a)<0$ for $j=1,\dots,n$, and $0<x_{n,j}(\alpha,\beta,\gamma;a)<1$ for $j=n+1,\dots,2n$. With this notation, \eqref{eq:ordering} means that
\begin{align*}
a   & <   x_{n,1}(\alpha,\beta,\gamma;a)   < x_{n-1,1}(\alpha+1,\beta+1,\gamma+1;a)< x_{n,2} (\alpha,\beta,\gamma;a) < \\ &   \dots < x_{n,n-1}(\alpha,\beta,\gamma;a)   < x_{n-1,n-1}(\alpha+1,\beta+1,\gamma+1;a) \\
& < x_{n,n}(\alpha,\beta,\gamma;a)<0.
\end{align*}
Since $P_n^{(\alpha,\beta,\gamma)}(x;a)$ and $P_{n-1}^{(\alpha+1,\beta+1,\gamma+1)}(x;a)$ are both of even degree with all their zeros on $(a,1)$, we get that
\begin{equation}\label{Signata}
    \sign \left(P_n^{(\alpha,\beta,\gamma)}(a;a)\right)=\sign \left(P_{n-1}^{(\alpha+1,\beta+1,\gamma+1)}(a;a)\right)=1 ,
\end{equation}
and that for $j=1,\dots,n-1$, 
\begin{equation}\label{notchangeofsign}
\begin{split}
\sign &\left(P_{n-1}^{(\alpha+1,\beta+1,\gamma+1)}(x_{n,j}(\alpha,\beta,\gamma;a);a)\right)   \\
   & = -\sign \left(P_n^{(\alpha,\beta,\gamma)}(x_{n-1,j}(\alpha+1,\beta+1,\gamma+1;a);a)\right)   \\
   & = -\sign   \left(P_{n-1}^{(\alpha+1,\beta+1,\gamma+1)}(x_{n,j+1}(\alpha,\beta,\gamma;a);a)\right)=(-1)^{j+1}.
\end{split}
\end{equation}

Recall that by \textit{(iii)} of Theorem \ref{identity},
$$
     P_n^{(\alpha,\beta,\gamma+1)}(x;a)=\mathfrak A_n P_n^{(\alpha,\beta,\gamma )}(x;a)+\mathfrak  B_n (x-a)(x-1)P_{n-1}^{(\alpha+1,\beta+1,\gamma+1)}(x;a)
    $$    
with $\mathfrak A_n   >0$, $\mathfrak  B_n   >0$ defined in \eqref{defParameters}. Since $(x-a)(x-1)<0$ on $(a,1)$, we have by \eqref{notchangeofsign} that for $j=1, \dots, n-1$,
 \begin{align*}
    \sign & (  P_{n}^{(\alpha,\beta,\gamma+1)}(x_{n-1,j }(\alpha+1,     \beta+1, \gamma+1;a);a) ) \\
    & = \sign\left(  P_{n}^{(\alpha,\beta,\gamma )}(x_{n-1,j }(\alpha+1,\beta+1,\gamma+1;a);a) \right),
 \end{align*}
 \begin{align*}
     \sign & \left( P_{n}^{(\alpha,\beta,\gamma+1)} (x_{n,j+1} (\alpha,\beta,\gamma;a);a) \right)   = -\sign\left( P_{n-1}^{(\alpha+1,\beta+1,\gamma+1)}(x_{n,j+1}(\alpha,\beta,\gamma;a);a)\right) \\
     & = -\sign \left(P_n^{(\alpha,\beta,\gamma)}(x_{n-1,j}(\alpha+1,\beta+1,\gamma+1;a);a)\right) .
 \end{align*}

In other words, $P_{n}^{(\alpha,\beta,\gamma+1)}(x;a)$ presents a sign change (and hence a root) in every subinterval  
$$
\left(  x_{n-1,j}(\alpha+1,\beta+1,\gamma+1;a), x_{n,j+1}(\alpha,\beta,\gamma;a) \right) , \quad j=1, \dots, n-1.
$$
Moreover, by \eqref{Signata},
 \begin{align*}
     \sign\left( P_{n}^{(\alpha,\beta,\gamma+1)}(a;a) \right)&=   \sign \left(P_n^{(\alpha,\beta,\gamma)}(a;a)\right) =1,
 \end{align*}
 while by \eqref{notchangeofsign},
  \begin{align*}
     \sign & \left( P_{n}^{(\alpha,\beta,\gamma+1)}(x_{n,1}(\alpha,\beta,\gamma;a) ;a) \right)\\
     &=   -\sign\left( P_{n-1}^{(\alpha+1,\beta+1,\gamma+1)}(x_{n,1}(\alpha,\beta,\gamma;a) ;a) \right) =-1,
 \end{align*}
 which means that another sign change (and the remaining root) of $P_{n}^{(\alpha,\beta,\gamma+1)}(x ;a)$ on $(a,0)$ is located in $(a, x_{n,1}(\alpha,\beta,\gamma;a))$.

In summary, 
\begin{align*}
a & < x_{n,1}(\alpha,\beta,\gamma+1;a)      <   x_{n,1}(\alpha,\beta,\gamma;a)    <   x_{n-1,1}(\alpha+1,\beta+1,\gamma+1;a) \\
& < x_{n,2}(\alpha,\beta,\gamma+1;a)       < x_{n,2} (\alpha,\beta,\gamma;a) < x_{n-1,2}(\alpha+1,\beta+1,\gamma+1;a) < \\ 
&   \dots < x_{n-1,n-1 }(\alpha+1,\beta+1,\gamma+1;a)   < x_{n ,n }(\alpha,\beta,\gamma+1;a) \\
& < x_{n ,n}(\alpha,\beta,\gamma;a)<0,
\end{align*}
which exactly means that 
$$
P_n^{(\alpha,\beta,\gamma+1)}(x;a)  \prec P_n^{(\alpha,\beta,\gamma)}(x;a) \prec P_{n-1}^{(\alpha+1,\beta+1,\gamma+1)}(x;a) \quad \text{on } (a,0),
$$
as stated. 

The proof of the interlacing on $(0,1)$ follows similar arguments, using that
\begin{equation*}
    \sign \left(P_n^{(\alpha,\beta,\gamma)}(0;a)\right)= - \sign \left( P_{n-1}^{(\alpha+1,\beta+1,\gamma+1)}(0;a)\right)=1. 
\end{equation*}
\end{proof}

\section{Mean behavior of the zeros } \label{sec:mean}

Here we gather some technical lemmas that will be useful in the next section.

\begin{lemma}[Geometric Mean]\label{geomean}
For $\alpha,\beta,\gamma>-1$, the geometric mean of the absolute values of the zeros of $P_n^{(\alpha,\beta,\gamma)}(x;a)$ is an increasing function of  $\gamma$.
\end{lemma}
\begin{proof}
By \eqref{explicitJA}, 
\begin{equation}\label{Geometric}
P_n^{(\alpha,\beta,\gamma)}(0;a) = \frac{d^{(n)}_{0,0}( \alpha,\beta,\gamma) }{c_n(\alpha,\beta,\gamma)} \, a^n = \binom{3n+\alpha+\beta+\gamma}{n}^{-1}\binom{n+\gamma}{n}\, a^n,
\end{equation}
so that
\begin{equation}\label{Geometric2}
\left| P_n^{(\alpha,\beta,\gamma)}(0;a)\right|=\frac{\Gamma(\alpha+\beta+\gamma+2n+1)\Gamma(n+\gamma+1)}{\Gamma(\alpha+\beta+\gamma+3n+1)\Gamma(\gamma+1)}\left| \, a\right|^n.
\end{equation}
Taking the logarithmic derivative of this identity with respect to $\gamma$ and using the following known identity \cite[Section 5.7(ii)]{NIST:DLMF},  
$$
\frac{\Gamma'(z)}{\Gamma(z)}=\psi(z)=-%
\gamma_e+\sum_{k=0}^{\infty}\left(\frac{1}{k+1}-\frac{1}{k+z}\right),
$$
(where $\gamma_e$ is Euler's constant) we obtain
\begin{align*}
\frac{\frac{\partial }{\partial \gamma}\left| P_n^{(\alpha,\beta,\gamma)}(0)\right|}{\left| P_n^{(\alpha,\beta,\gamma)}(0)\right|}  = & \sum_{k=0}^{\infty}\left( \frac{1}{k+1}-\frac{1}{k+\alpha+\beta+\gamma+2n+1}\right) \\
& -\sum_{k=0}^{\infty}\left( \frac{1}{k+1}- \frac{1}{k+\alpha+\beta+\gamma+3n+1} \right)  \\ &
+\sum_{k=0}^{\infty}\left( \frac{1}{k+1}-\frac{1}{k+n+\gamma+1}\right) \\ & - \sum_{k=0}^{\infty}\left( \frac{1}{k+1}-\frac{1}{k+\gamma+1}\right).
\end{align*}
After simplification, we arrive at
\begin{align*}
\frac{\frac{\partial }{\partial \gamma}\left| P_n^{(\alpha,\beta,\gamma)}(0)\right|}{\left| P_n^{(\alpha,\beta,\gamma)}(0)\right|} & =\sum_{k=1}^\infty \left( \frac{1}{k+\gamma}-\frac{1}{k+\alpha+\beta+\gamma+2n} \right)\\
& =\sum_{k=1}^\infty\frac{\alpha+\gamma+2n}{(k+\gamma)(k+\alpha+\beta+\gamma+2n)}.
\end{align*}
Since by assumption, $\alpha,\beta,\gamma>-1$,  each  term in the sum on the right-hand side is positive, so that
\begin{equation*}
    \frac{\partial}{\partial \gamma}\left| P_n^{(\alpha,\beta,\gamma)}(0)\right|>0.
\end{equation*}
\end{proof}

\begin{corollary}
    \label{limitsofzeros}
   Under the assumptions above, and with notation \eqref{def:Zeros}, 
    \begin{equation} \label{limits}
    \begin{split}
        \lim_{\gamma\to+\infty}x_{n,j}(\alpha,\beta,\gamma;a)=a,& \quad j=1,\dots,n\\
        \lim_{\gamma\to +\infty}x_{n,j}(\alpha,\beta,\gamma;a)=1,&\quad j= n+1,\dots,2n.
    \end{split}
    \end{equation}
\end{corollary}
\begin{proof}
 Since  $| x_{n,j}(\alpha,\beta,\gamma;a)|< |a|$ for $j=1,\dots,n$, and $| x_{n,j}(\alpha,\beta,\gamma;a)|< 1$ for $j=n+1,\dots,2n$, we conclude that 
 \begin{equation} \label{bound1}
     \prod_{j=1}^{2n}  \left|x_{n,j}(\alpha,\beta,\gamma;a)\right|\leq \left| a^n\right|.
 \end{equation}
On the other hand, taking the limit in \eqref{Geometric2} as $\gamma\to +\infty$ and using Stirling's formula \cite[Formula 5.11.E7]{NIST:DLMF}, we get that 
\begin{equation*}
        \lim_{\gamma\to +\infty}\left|P_n^{(\alpha,\beta,\gamma)}(0;a)\right|=\lim_{\gamma\to+ \infty}\prod_{j=1}^{2n} \left|x_{n,j}(\alpha,\beta,\gamma;a)\right|=\left|a^n\right|.
    \end{equation*}
Combining it with the upper bound \eqref{bound1} we obtain \eqref{limits}.
\end{proof}
 
Our next result concerns the arithmetic mean of the zeros of $P_n^{(\alpha,\beta,\gamma)}(x;a)$, 
$$
A_n^{(\alpha,\beta,\gamma)}(a):=\frac{1}{2n}\, \sum_{j=1}^{2n}x_{n,j}(\alpha,\beta,\gamma;a).
$$

\begin{lemma}[Arithmetic Mean]\label{aritmetic}
For $\alpha,\beta,\gamma>-1$, 
\begin{equation}
    \label{eq:arithmeticmean}
    A_n^{(\alpha,\beta,\gamma)}(a) = \frac{1}{2} \frac{(2n+\beta+\gamma) + a\left(2n+\gamma+\alpha \right) }{3n+\alpha+\beta+\gamma} .
\end{equation}
\end{lemma}
\begin{proof}
Using \eqref{explicitJA} and the relation between the arithmetic mean of the zeros and the coefficient of $x^{2n-1}$ of $P_n^{(\alpha,\beta,\gamma)}(x)$ we get that
\begin{equation} \label{arithmeticMean1}
c_n(\alpha,\beta,\gamma)A_n^{(\alpha,\beta,\gamma)}(a)=\frac{1}{2n}\sum_{k=0}^n\sum_{j=0}^{n-k} d^{(n)}_{j,k}( \alpha,\beta,\gamma)\left((n-k)+(n-j)a\right). 
\end{equation}
By the identity
$$
\sum_{j=0}^m\binom{a}{j}\binom{b}{m-j}=\binom{a+b}{m} 
$$
we have that
\begin{align*}
    \sum_{k=0}^n\sum_{j=0}^{n-k} d^{(n)}_{j,k}( \alpha,\beta,\gamma)(n-k) & =n\frac{2n+\beta+\gamma}{3n+\alpha+\beta+\gamma}\binom{3n+\alpha+\beta+\gamma}{n},
\\
    \sum_{k=0}^n\sum_{j=0}^{n-k} d^{(n)}_{j,k}( \alpha,\beta,\gamma)(n-j) &=n\frac{2n+\gamma+\alpha}{3n+\alpha+\beta+\gamma}\binom{3n+\alpha+\beta+\gamma}{n}.
\end{align*}
Combining these two formulas in \eqref{arithmeticMean1} finally yields \eqref{eq:arithmeticmean}.
\end{proof}

\begin{corollary} \label{limit:arithmetic}
For $\alpha,\beta >-1$ and $n\in \N$,
     \begin{equation}\label{eq;derivativearit}
   \lim_{\gamma\to +\infty} \frac{\partial}{\partial\gamma} A_n^{(\alpha,\beta,\gamma)}(a)=
        \lim_{\gamma\to +\infty}\sum_{j=1}^{2n}\frac{\partial}{\partial\gamma}x_{n,j}(\alpha,\beta,\gamma;a)=0.
    \end{equation}
\end{corollary}
\begin{proof}
This is an immediate consequence of differentiating \eqref{eq:arithmeticmean}  with respect to $\gamma$:
  $$  
    \frac{\partial}{\partial\gamma}A^{(\alpha,\beta,\gamma)}(a)  =    \frac{1}{2n}\, \sum_{j=1}^{2n}\frac{\partial}{\partial\gamma}x_{n,j}(\alpha,\beta,\gamma;a)= \frac{n+\alpha+ (n+\beta)a}{2(3n+\alpha+\beta+\gamma)^2} .
$$  
\end{proof}

\section{Proof of Theorem \ref{thm:monotonicity}} \label{sec:Monotonicity}

One of the main tools in the proof of Theorems \ref{dossantosformula} and \ref{dossantosformula2}
in \cite{dos2017monotonicity} is the function (we preserve the notation of \cite{dos2017monotonicity})
\begin{multline}\label{auxiliaryfunction}
     f_{n-1}(x,a,\alpha+1,\beta+1,\gamma+1) :=\\  \sum_{j=1}^{2(n-1)}\frac{1}{x-x_{n-1,j}(a,\alpha+1,\beta+1,\gamma+1)} 
   +\frac{\beta+1}{x-a}+\frac{\gamma+1}{x}-\frac{\alpha+1}{1-x},
\end{multline}
where, as in \eqref{def:Zeros}, $x_{n,j}(\alpha,\beta,\gamma;a)$  are the zeros of $P^{(\alpha,\beta,\gamma)}_{n }(x)$. As it was shown in \cite[formula (2.7)]{dos2017monotonicity}, this can be written as an irreducible fraction with the numerator proportional to $P^{(\alpha,\beta,\gamma)}_{n }(x)$, so that    
\begin{equation}
   \label{DesanosF}
   f_{n-1}(x,a,\alpha+1,\beta+1,\gamma+1)=0 \quad \text{for } x= x_{n,j}(\alpha, \beta,\gamma; a), \quad j=1, \dots, 2n.
\end{equation}

\begin{proof}[Proof of Theorem \ref{thm:monotonicity}]
With a shorter notation $x_{n,j}(\gamma)$ instead of $x_{n,j}(\alpha,\beta,\gamma;a)$, $y_{n-1,j}(\gamma)$ instead of $x_{n-1,j}(\alpha+1,\beta+1,\gamma+1;a)$, and $f_{n-1}(x; \gamma+1)$ instead of $ f_{n-1}(x,a,\alpha+1,\beta+1,\gamma+1)$, we use implicit differentiation in \eqref{DesanosF} to obtain
\begin{equation}
    \label{eq:derivative1}
    \frac{\partial}{\partial\gamma}\, x_{n,j}(\gamma) =-\frac{\frac{\partial}{\partial \gamma}f_{n-1}(x_{n,j}(\gamma) ;\gamma+1)}{\frac{\partial}{\partial x}f_{n-1}(x_{n,j} (\gamma);\gamma+1)}, \qquad j=1, \dots, 2n.
\end{equation}
In \cite{dos2017monotonicity} it was shown that $f_{n-1}(x;  \gamma+1)$ is a strictly decreasing function of $x$ in each open interval limited by consecutive zeros of $x(x-a)(x-1)P^{(\alpha+1,\beta+1,\gamma+1)}_{n-1}(x;a)$, so that by \eqref{eq:derivative1} and $(ii)$ of Lemma \ref{dossantosformula}, 
$$
\sign\left(\frac{\partial}{\partial\gamma}x_{n,j}(\gamma) \right)=\sign\left( \frac{\partial}{\partial \gamma}f_{n-1}(x_{n,j}(\gamma);\gamma+1)\right), \quad j=1, \dots, 2n.
$$ 
Thus, we can prove the theorem by analyzing the sign of the derivative $\frac{\partial}{\partial \gamma}f_{n-1}(x;\gamma+1)$, evaluated at the zeros $x_{n,j}(\gamma)$ of $P^{(\alpha,\beta,\gamma)}_{n }(x)$: we need to show that
\begin{equation*}  
    \begin{split}
    \frac{\partial}{\partial\gamma}f_{n-1}( x_{n,j}(\gamma) ;\gamma+1)&<0\quad \text{for } j=1,\dots,n, \\
    \frac{\partial}{\partial\gamma}f_{n-1}( x_{n,j}(\gamma);\gamma+1)&>0\quad \text{for } j=n+1,\dots,2n.
    \end{split}
\end{equation*}
Observe that
\begin{equation}
    \label{derivativeOfF}
    \frac{\partial}{\partial\gamma}f_{n-1}(x;\gamma+1)=\begin{cases}
    \dfrac{1}{x}, & n=1, \\
        \sum_{j=1}^{2(n-1)}\dfrac{ \frac{\partial}{\partial \gamma}(y_{n-1,j}(\gamma))}{(x-y_{n-1,j}(\gamma))^2}+\dfrac{1}{x}, & n\ge 2.
    \end{cases}
\end{equation}

We proceed by induction on $n$. The case of $n=1$  is trivial: by \eqref{derivativeOfF},
    \begin{equation*}
        \frac{\partial}{\partial\gamma}f_{n-1}(x_{1,j}(\gamma);\gamma+1)=\frac{1}{x_{1,j}(\gamma)}\begin{cases}
             <0 & \text{if } j=1,\\
              >0 & \text{if } j=2,
        \end{cases}
    \end{equation*}
as needed. 

Assume that for an $n\in \N$, $n\ge 2$, we have that there exists $\Gamma_1>-1$ such that for $\gamma>\Gamma_1$, 
\begin{equation} \label{negativedr}
\begin{split}
    \frac{\partial}{\partial\gamma}\, y_{n-1,j}(\gamma)<0& \quad \text{for } j=1,\dots,n-1, \\
     \frac{\partial}{\partial\gamma} \, y_{n-1,j}(\gamma)>0& \quad \text{for } j=n,\dots,2(n-1)  .   
\end{split}
\end{equation}
By Corollary \ref{limitsofzeros}, there exists $\Gamma_2>-1$ such that for $\gamma> \Gamma_2$, 
\begin{equation}
\label{bounds}
    \begin{split}
        a<x_{n,i}(\gamma)<\frac{1}{2}a  & \qquad \text{for } i=1,\dots,n, \\
        \frac{1}{2}<x_{n,i}(\gamma)<1  & \qquad \text{for } j=n+1,\dots,2n;
    \end{split}
\end{equation}
notice that by Theorem \ref{dossantosformula}, we also have that for $\gamma> \Gamma_2$, 
\begin{equation}
\label{bounds2}
    \begin{split}
        a<y_{n-1,i}(\gamma)<\frac{1}{2}a  & \qquad \text{for } i=1,\dots,n-1, \\
        \frac{1}{2}<y_{n-1,i}(\gamma)<1  & \qquad \text{for } j=n+1,\dots,2n-2.
    \end{split}
\end{equation}
Finally,  by Corollary \ref{limit:arithmetic} and since $a<0$, there exists $\Gamma_3>-1$ such that for $\gamma> \Gamma_3$, 
\begin{equation} \label{boundAverageDr}
    -\frac{(1-a)^2}{4}<\sum_{j=1}^{2(n-1)} \frac{\partial}{\partial\gamma}(y_{n-1,j}(\gamma))<-\frac{(1-a)^2}{4a}.
\end{equation}
Set $\Gamma= \max \left\{\Gamma_1, \Gamma_2, \Gamma_3 \right\}$ and assume $\gamma>\Gamma$. By \eqref{boundAverageDr},
\begin{align}
\label{ineq1}
    \frac{4}{(1-a)^2} \sum_{j=n}^{2(n-1)} \frac{\partial}{\partial\gamma}(y_{n-1,j}(\gamma)) &<-\frac{1}{a}- \frac{4}{(1-a)^2}\sum_{j=1}^{ n-1} \frac{\partial}{\partial\gamma}(y_{n-1,j}(\gamma)) ,
    \\
    \label{ineq2}
 \frac{4}{(1-a)^2}\sum_{j=1}^{n-1}\frac{\partial}{\partial \gamma}(y_{n-1,j}(\gamma)) & 
 > -1- \frac{4}{(1-a)^2}\sum_{j=n}^{2(n-1)}\frac{\partial}{\partial \gamma}(y_{n-1,j}(\gamma)) . 
\end{align}

Using \eqref{negativedr}--\eqref{bounds2} we conclude that for $i=1, 2, \dots, n$,
\begin{align}
    \label{Bound:negative1}
    \sum_{j=1}^{n-1}\frac{ \frac{\partial}{\partial \gamma}(y_{n-1,j}(\gamma))}{(x_{n,i}(\gamma)-y_{n-1,j}(\gamma))^2} & <\frac{4}{a^2}\sum_{j=1}^{n-1} \frac{\partial}{\partial \gamma}(y_{n-1,j}(\gamma)),
\end{align}
as well as
\begin{equation}
    \label{Bound:negative3}
   \begin{split}
    \sum_{j=n}^{2(n-1)}\frac{ \frac{\partial}{\partial \gamma}(y_{n-1,j}(\gamma))}{(x_{n,i}(\gamma)-y_{n-1,j}(\gamma))^2} & <\frac{4}{(1-a)^2}\sum_{j=n}^{2(n-1)}\frac{\partial}{\partial \gamma}(y_{n-1,j}(\gamma))\\
 &< -\frac{1}{a}- \frac{4}{(1-a)^2}\sum_{j=1}^{n-1}\frac{\partial}{\partial \gamma}(y_{n-1,j}(\gamma)), 
   \end{split} 
\end{equation}
where we have used \eqref{ineq1}. Combining \eqref{Bound:negative1} and \eqref{Bound:negative3} in \eqref{derivativeOfF} and taking into account that $a<0$ and $1/x_{n,i}(\gamma)< 1/a$, we get 
that for $i=1, 2, \dots, n$,
$$
\frac{\partial}{\partial\gamma}f_{n}(x_{n,i}(\gamma);\gamma)<4\left( \frac{1}{a^2}-\frac{1}{(1-a)^2}\right)\sum_{j=1}^{n-1}\frac{\partial}{\partial\gamma}(y_{n-1,j}(\gamma))-\frac{1}{a}+\frac{1}{x_{n,i}(\gamma)}<0.
$$

In the same fashion, for $i=n+1, \dots, 2n$, from \eqref{negativedr}--\eqref{bounds2} we conclude that  
\begin{equation}\label{Bound: positve2}
    \sum_{j=n}^{2(n-1)}\frac{ \frac{\partial}{\partial \gamma}(y_{n-1,j}(\gamma))}{(x_{n,i}(\gamma)-y_{n-1,j}(\gamma))^2}>4\sum_{j=n}^{2(n-1)} \frac{\partial}{\partial \gamma}(y_{n-1,j}(\gamma)),
\end{equation}
as well as
\begin{multline}\label{Bound:negative2}
 \sum_{j=1}^{n-1}\frac{ \frac{\partial}{\partial \gamma}(y_{n-1,j}(\gamma))}{(x_{n,i}(\gamma)-y_{n-1,j}(\gamma))^2}>\frac{4}{(1-a)^2}\sum_{j=1}^{n-1}\frac{\partial}{\partial \gamma}(y_{n-1,j}(\gamma))\\
 > -1- \frac{4}{(1-a)^2}\sum_{j=n}^{2(n-1)}\frac{\partial}{\partial \gamma}(y_{n-1,j}(\gamma)),
\end{multline} 
where we have used \eqref{ineq2}. Combining  \eqref{Bound:negative2} and \eqref{Bound: positve2} in \eqref{derivativeOfF} and taking into account that $a<0$ and $1/x_{n,i}(\gamma)> 1$, we get finally that for $i=n+1, \dots, 2n$, 
$$
\frac{\partial}{\partial\gamma}f_{n}(x_{n,i}(\gamma);\gamma)> 4 \left( 1-\frac{1}{(1-a)^2}\right)\sum_{j=n}^{2(n-1)}\frac{\partial}{\partial\gamma}(y_{n-1,j}(\gamma))-1+\frac{1}{x_{n,i}(\gamma)}>0.
$$
\end{proof}

\section{Proof of Theorem \ref{thm:monotonicityA}} \label{sec:monotA}

We need the following result to understand the behavior of the zeros of $P_n^{(\alpha,\beta,\gamma)}(x;a)$ as functions of the endpoint $a$:
\begin{lemma}
\label{identityn}
Let $n\in\mathbb{N}$, and $\alpha,\beta, \gamma>-1$. Then
\begin{align*}
& (x-a)\frac{\partial}{\partial a}   P_n^{(\alpha,\beta+1,\gamma)}(x;a)  \\
& = \frac{-(n+\beta+1)(1+\alpha+\beta+\gamma+2n)}{1+\alpha+\beta+\gamma+3n} P_n^{(\alpha,\beta ,\gamma)}(x;a)+(\beta+1) P_{n}^{(\alpha,\beta+1,\gamma)}(x;a) \\
& =- n P_n^{(\alpha,\beta+1,\gamma)}(x;a)   +\frac{n(n+\beta+1)}{1+\alpha+\beta+\gamma+3n} x(x-1)P_{n-1}^{(\alpha+1,\beta+1,\gamma+1)}(x;a).
\end{align*}
\end{lemma}
\begin{proof}
We prove the first identity by using the Rodrigues formula \eqref{RodriguezForm}: differentiating it with respect to $a$,  the right-hand side yields 
\begin{equation}
    \begin{split}\label{leftside}
        \frac{\partial}{\partial a}\biggl((1-x)^\alpha(x-a)^\beta x^\gamma c_n(\alpha,\beta,\gamma) & P^{(\alpha,\beta,\gamma)}_{n}(x;a)\biggr)\\
        = (1-x)^{\alpha} (x-a)^{\beta} x^{\gamma}&c_n(\alpha,\beta,\gamma)\frac{\partial}{\partial a}  P_n^{(\alpha,\beta,\gamma)}(x;a)\\ -\beta (1-x)^{\alpha} (x-a)^{\beta-1} &x^{\gamma}c_n(\alpha,\beta,\gamma) P_n^{(\alpha,\beta,\gamma)}(x;a),
    \end{split}
\end{equation}
and on the left-hand side, 
\begin{align}
 \nonumber
   \frac{(-1)^n}{n!}\frac{\partial}{\partial a}\Biggl( &\left( \frac{d}{dx} \right)^n\left[(1-x)^{\alpha+n}(x-a)^{\beta+n}x^{\gamma+n}\right]\Biggr)\\  
    & =(n+\beta)\frac{(-1)^{n+1}}{n!} \left( \frac{d}{dx} \right) ^n\left[ (1-x)^{\alpha+n}(x-a)^{\beta+n-1}x^{\gamma+n}\right] \nonumber \\
  &  =-(n+\beta)(1-x)^\alpha(x-a)^{\beta-1} x^\gamma c_n(\alpha,\beta-1,\gamma) P^{(\alpha,\beta-1,\gamma)}_{n}(x;a).
  \label{Rightside}
\end{align}

Combining \eqref{leftside} and \eqref{Rightside} and diving through by $ (1-x)^{\alpha}(x-a)^{\beta-1}x^{\gamma}$ we obtain the following equation, 
\begin{multline*}
    (x-a)c_n(\alpha,\beta,\gamma)\frac{\partial}{\partial a}P_n^{(\alpha,\beta,\gamma)}(x;a)-\beta c_n(\alpha,\beta,\gamma) P_n^{(\alpha,\beta,\gamma)}(x;a)\\=-(n+\beta)c_n(\alpha,\beta-1,\gamma) P^{(\alpha,\beta-1,\gamma)}_{n}(x;a),
\end{multline*}
which is the first equality claimed in the lemma. 

The second identity can be derived from the explicit expression \eqref{explicitJA}. 
Differentiating it with respect to $a$, we get
\begin{align*}
c_n(\alpha,\beta+1,\gamma)& \frac{\partial}{\partial a} P_n^{(\alpha,\beta+1,\gamma)}(x;a) \\
& =-\sum_{k=0}^n\sum_{j=0}^{n-k}(n-j)d^{(n)}_{j,k}( \alpha,\beta+1,\gamma)(x-1)^{n-k}x^{k+j}(x-a)^{n-1-j},
\end{align*}
so that 
\begin{equation}\label{separetedsums}
\begin{split}
    c_n(\alpha,\beta+1,\gamma)&(a-x)\frac{\partial}{\partial a}  P_n^{(\alpha,\beta+1,\gamma)}(x;a) \\
    & = n\sum_{k=0}^n\sum_{j=0}^{n-k}d^{(n)}_{j,k}( \alpha,\beta+1,\gamma)(x-1)^{n-k}x^{k+j}(x-a)^{n-j}\\
  &  -\sum_{k=0}^{n-1}\sum_{j=1}^{n-k}jd^{(n)}_{j,k}( \alpha,\beta+1,\gamma)(x-1)^{n-k}x^{k+j}(x-a)^{n-j}.
\end{split}
\end{equation}

By \eqref{explicitJA},
\begin{equation}\label{firstterm}
\begin{split}
    nc_n(\alpha,\beta+1,\gamma) & P_n^{(\alpha,\beta+1,\gamma)}(x;a) \\ 
    & =n\sum_{k=0}^n\sum_{j=0}^{n-k}d^{(n)}_{j,k}( \alpha,\beta+1,\gamma)(x-1)^{n-k}x^{k+j}(x-a)^{n-j}.
    \end{split}
\end{equation}
On the other hand, by \eqref{explicitCoefficientsJA}, for $j>0$ we have 
\begin{equation*}
    jd^{(n)}_{j,k}(\alpha,\beta+1,\gamma)=(n+\beta+1)d^{(n-1)}_{j-1,k}(\alpha+1,\beta+1,\gamma+1),
\end{equation*}
which implies that  
\begin{equation*}
\begin{split}
    \sum_{k=0}^{n-1} & \sum_{j=1}^{n-k}jd^{(n)}_{j,k}( \alpha,\beta+1,\gamma)(x-1)^{n-k}x^{k+j}(x-a)^{n-j} = (n+\beta+1) \times \\ &\sum_{k=0}^{n-1}\sum_{j=0}^{n-1-k}d^{(n-1)}_{j,k}( \alpha+1,\beta+1,\gamma+1)(x-1)^{n-k}x^{k+j+1}(x-a)^{n-1-j},
\end{split}
\end{equation*}
or equivalently, by \eqref{explicitJA}, 
\begin{equation}\label{secondterm}
\begin{split}
    \sum_{k=0}^{n-1} & \sum_{j=1}^{n-k}jd^{(n)}_{j,k}( \alpha,\beta+1,\gamma)(x-1)^{n-k}x^{k+j}(x-a)^{n-j}=\\
    &(n+\beta+1)c_{n-1}(\alpha+1,\beta+1,\gamma+1)x(x-1)P_{n-1}^{(\alpha+1,\beta+1,\gamma+1)}(x;a).
\end{split}
\end{equation}
Using the equations \eqref{firstterm} and \eqref{secondterm} in \eqref{separetedsums}, we get
\begin{equation*}
\begin{split}
    & c_n(\alpha,\beta+1,\gamma)(a-x)\frac{\partial}{\partial a}  P_n^{(\alpha,\beta+1,\gamma)}(x;a)= n c_n(\alpha,\beta+1,\gamma) P_n^{(\alpha,\beta+1,\gamma)}(x;a)\\
    &-(n+\beta+1)c_{n-1}(\alpha+1,\beta+1,\gamma+1)x(x-1)P_{n-1}^{(\alpha+1,\beta+1,\gamma+1)}(x;a),
\end{split}
\end{equation*}
which is equivalent to the second identity of the theorem.
\end{proof}

\medskip

\begin{proof}[Proof of Theorem \ref{thm:monotonicityA}]
From the implicit differentiation of 
     $$
     P_n^{(\alpha,\beta+1,\gamma)}(x_{n,j}(\alpha,\beta+1,\gamma;a);a)=0
     $$
     with respect to the parameter $a$, we get 
     \begin{equation*}
          \frac{\partial}{\partial a}x_{n,j}(\alpha,\beta+1,\gamma;a)=-\frac{\frac{\partial}{\partial a}P_{n}^{(\alpha,\beta+1,\gamma)}(x_{n,j}(\alpha,\beta+1,\gamma;a);a)}{\frac{\partial}{\partial x}P_{n}^{(\alpha,\beta+1,\gamma)}(x_{n,j}(\alpha,\beta+1,\gamma;a);a)}
     \end{equation*}
 Using Lemma \ref{identityn} to replace the numerator on the right-hand side, we obtain
\begin{multline} \label{finalidentityXroots}
         \frac{\partial}{\partial a}x_{n,j}(\alpha,\beta+1,\gamma;a)= \frac{(n+\beta+1)(1+\alpha+\beta+\gamma+2n)}{1+\alpha+\beta+\gamma+3n} \times \\ \frac{ P_n^{(\alpha,\beta ,\gamma)}(x_{n,j}(\alpha,\beta+1,\gamma;a);a)}{\left(x_{n,j}(\alpha,\beta+1,\gamma;a)-a\right)\frac{\partial}{\partial x}P_{n}^{(\alpha,\beta+1,\gamma)}(x_{n,j}(\alpha,\beta+1,\gamma;a);a)}.
 \end{multline}
Recall that 
$x_{n,j}(\alpha,\beta+1,\gamma;a) >a$ for all $j=1,\dots,2n$. Furthermore, $P_{n}^{(\alpha,\beta+1,\gamma)}(x ;a) $ is a polynomial of degree $2n$, with $P_{n}^{(\alpha,\beta+1,\gamma)}(a;a)>0$ and simple zeros on $(a,1)$, which implies that
 $$
  \sign\left(\frac{\partial}{\partial x}P_{n}^{(\alpha,\beta+1,\gamma)}(x_{n,j}(a,\alpha,\beta+1,\gamma;a)\right) =(-1)^j, \quad j=1, \dots, 2n.
$$
Analogously, $P_{n}^{(\alpha,\beta+1,\gamma)}(x;a)$ is a polynomial of degree $2n$, with $P_{n}^{(\alpha,\beta+1,\gamma)}(a;a)>0$. The assertion \textit{(ii)} of Theorem \ref{entrelazamiento} implies that
    $$
    \sign\left(P_{n}^{(\alpha,\beta,\gamma)}(x_{n,j}(\alpha,\beta+1,\gamma;a);a)\right)=(-1)^j, \quad j=1, \dots, 2n. 
     $$
Gathering these facts in \eqref{finalidentityXroots} yields
$$
\frac{\partial}{\partial a}x_{n,j}(\alpha,\beta+1,\gamma;a)>0, \quad j=1, \dots, 2n.
$$
\end{proof}

\section{Proof of Theorem \ref{thm:interlacingLaguerre}} \label{sec:Laguerre}

The results obtained for the Angelesco-Jacobi polynomials can be extended to other families of multiple orthogonal polynomials of type II, namely, the Jacobi-Laguerre and Laguerre-Hermite polynomials. Indeed, due to the analyticity of the functions involved, we can take appropriate limits in \eqref{raisingJA} (the fact that the limits and differentiation commute can be easily justified using Cauchy's formula), and \eqref{LaguerreLimit}--\eqref{HermiteLimit} yield the following raising operator identities:
\begin{equation}\label{raisingJL}
           (x-a)^\beta x^\gamma  e^{-x} L^{(\beta,\gamma)}_{n}(x;a)
         =-\frac{d}{dx} \left[(x-a)^{\beta+1}x^{\gamma+1}e^{-x} L^{(\beta+1,\gamma+1)}_{n-1}(x;a)\right],
\end{equation}
\begin{equation}\label{raisingHL}
            x^\gamma  e^{-x^2} H^{(\gamma)}_{n}(x)
         =-\frac{1}{2}\frac{d}{dx} \left[x^{\gamma+1}e^{-x^2} H^{(\gamma+1)}_{n-1}(x)\right]. 
\end{equation}

\begin{remark}\label{nocommonzeros}
    Since the roots of $L^{(\beta+1,\gamma+1)}_{n-1}(x;a)$ are simple, using \eqref{raisingJL}-\eqref{raisingHL}, we conclude that $L^{(\beta+1,\gamma+1)}_{n-1}(x;a)$ and $L^{(\beta,\gamma)}_{n}(x;a)$ do not have common zeros. The same assertion is valid for the polynomials $H^{(\gamma+1)}_{n-1}(x)$ and $H^{(\gamma)}_{n}(x)$.
\end{remark}
The following generalization of the structure relations, obtained in Theorem \ref{identity}, take place:
\begin{theorem}\label{identityJl} Let $n\in\mathbb{N}$ and $\beta,\gamma>-1$. Then 
\begin{enumerate}
    \item 
    $$
     L_n^{(\beta+1,\gamma)}(x;a)=  L_n^{(\beta,\gamma)}(x;a)- xL_{n-1}^{(\beta+1,\gamma+1)}(x;a).
    $$
    \item  
    $$
     L_n^{(\beta,\gamma+1)}(x;a)=  L_n^{(\beta,\gamma)}(x;a)- (x-a)L_{n-1}^{(\beta+1,\gamma+1)}(x;a).
    $$
\end{enumerate}
\end{theorem}
\begin{proof}
By $(ii)$ of Theorem \ref{identity},
\begin{equation}\label{LaguerreH1}
    \begin{split}
        \alpha^{2n}P_n^{(\alpha,\beta+1,\gamma)}&(x/\alpha;a/\alpha)= \mathfrak A_n\alpha^{2n} P_n^{(\alpha,\beta,\gamma)}(x/\alpha;a/\alpha)+\\ &\mathfrak B_n\alpha^{2n-1} (x/\alpha-1)xP_{n-1}^{(\alpha+1,\beta+1,\gamma+1)}(x/\alpha;a/\alpha),
    \end{split}
\end{equation}
where $\mathfrak A_n = \mathfrak A_n(\alpha,\beta,\gamma)$ and $\mathfrak B_n=\mathfrak B_n(\alpha,\beta,\gamma)= 1- \mathfrak A_n $ were defined in \eqref{defParameters}. From their explicit expression,  
    $$
    \lim_{\alpha\to\infty}\mathfrak A_n(\alpha,\beta,\gamma)=\lim_{\alpha\to\infty} \alpha \mathfrak B_n(\alpha,\beta,\gamma)=1.
    $$
    Thus, taking the limit in \eqref{LaguerreH1} as $\alpha \to \infty$ we get assertion \textit{(1)} of the theorem. Similar arguments, now with assertion \textit{(iii)} of Theorem \ref{identity}, yield \textit{(2)}. 
\end{proof}

\begin{theorem}\label{identityLH} Let $n\in\mathbb{N}$ and $\gamma>-1$. Then 
\begin{equation}
    H_n^{(\gamma+1)}(x)=  H_n^{(\gamma)}(x;a)- \frac{1}{2}H_{n-1}^{(\gamma+1)}(x;a).
\end{equation}
\end{theorem}
\begin{proof}
    By $(iii)$ of Theorem \ref{identity},
    \begin{multline}\label{Hermite}
        \alpha^n P_n^{(\alpha,\alpha,\gamma+1)}(x/\sqrt{\alpha};-1)=  \mathfrak A_n \alpha^n P_n^{(\alpha,\alpha,\gamma)}(x/\sqrt{\alpha};-1) + \\ \mathfrak B_n (x/\sqrt{\alpha}+1)(x/\sqrt{\alpha}-1)\alpha^n P_{n-1}^{(\alpha+1,\alpha+1,\gamma+1)}(x/\sqrt{\alpha};-1), 
    \end{multline}
    with $\mathfrak A_n = \mathfrak A_n(\alpha,\beta,\alpha)$ and $\mathfrak B_n(\alpha,\beta,\alpha)= 1- \mathfrak A_n$.  
    By \eqref{defParameters}.
    $$
    \lim_{\alpha\to\infty}\mathfrak A_n(\alpha,\beta,\alpha)=1 \quad \text{and }\quad\lim_{\alpha\to\infty} \alpha \mathfrak B_n(\alpha,\beta,\alpha)=1/2, 
    $$
  and it remains to take the limit in \eqref{Hermite} as $\alpha\to \infty$. 
\end{proof}

\begin{proof}[Proof of Theorem \ref{thm:interlacingLaguerre}]
Part \textit{\eqref{thm21ii}} of Theorem \ref{entrelazamiento} can be restated as
\begin{equation}
\begin{split}
   \alpha^{2n} P_n^{(\alpha,\beta,\gamma)}(x/\alpha;a/\alpha) &     \prec\alpha^{2n} P_n^{(\alpha,\beta+1,\gamma)}(x/\alpha;a/\alpha) \\ &\prec x(x/\alpha-1)\alpha^{2n-2}P_{n-1}^{(\alpha+1,\beta+1,\gamma+1)}(x/\alpha;a/\alpha).
\end{split}
\end{equation}
With the notation \eqref{def:Zeros}, this means that
\begin{equation}\label{Auxiliarinterlacing}
\begin{split}
a   & <   x_{n,1}(\alpha,\beta,\gamma;a/\alpha)   < x_{n,1}(\alpha,\beta+1,\gamma;a/\alpha)  
\\
 &  < x_{n-1,1} (\alpha+1,\beta+1,\gamma+1;a/\alpha)  <\dots < x_{n-1,n-1}(\alpha+1,\beta+1,\gamma+1;a/\alpha)    \\
& < x_{n,n}(\alpha,\beta,\gamma;a/\alpha) < x_{n,n}(\alpha,\beta+1,\gamma;a/\alpha)<0 <   x_{n,n+1}(\alpha,\beta,\gamma;a/\alpha) \\ 
&  < x_{n,n+1}(\alpha,\beta+1,\gamma;a/\alpha)< x_{n-1,n} (\alpha+1,\beta+1,\gamma+1;a/\alpha)  \\  &  < \dots < x_{n-1,2n-2}(\alpha+1,\beta+1,\gamma+1;a/\alpha)   < x_{n,2n}(\alpha,\beta,\gamma;a/\alpha) \\
&< x_{n,2n}(\alpha,\beta+1,\gamma;a/\alpha)<1.
\end{split}
\end{equation}

Let us denote by 
\begin{equation}
    \label{def:ZerosL}
    l_{n,1}(\beta,\gamma;a) < l_{n,2}(\beta,\gamma;a) < \dots < l_{n,2n}(\beta,\gamma;a) 
\end{equation}
the zeros of $L^{(\beta,\gamma)}_n(x;a)$. Recall that $a<l_{n,j}(\beta,\gamma;a)<0$ for $j=1,\dots,n$ and $0<l_{n,j}(\beta,\gamma;a)$ for $j=n+1,\dots,2n$. Also, by \eqref{LaguerreLimit},  
\begin{equation}\label{JatoJL}
    \lim_{\alpha\to\infty} \alpha x_{n,j}(\alpha,\beta,\gamma;a/\alpha)=l_{n,j}(\beta,\gamma;a) \quad j=1,\dots,2n.
\end{equation}

Multiplying \eqref{Auxiliarinterlacing} by $\alpha$ and taking the limit as $\alpha\to \infty$ we get
\begin{equation}
    \label{interlacingL}
    \begin{split}
a   & <  l_{n,1}(\beta,\gamma;a)   \leq l_{n,1}(\beta+1,\gamma;a)\leq l_{n-1,1} (\beta+1,\gamma+1;a) \leq \\  &   \dots \leq l_{n-1,n-1}(\beta+1,\gamma+1;a)   \leq l_{n,n}(\beta,\gamma;a)  \leq l_{n,n}(\beta+1,\gamma;a)<0 \\ 
&\leq   l_{n,n+1}(\beta,\gamma;a)   \leq l_{n,n+1}(\beta+1,\gamma;a)\leq l_{n-1,n} (\beta+1,\gamma+1;a) \leq \\  &   \dots \leq l_{n-1,2n-2}(\beta+1,\gamma+1;a)   \leq l_{n,2n}(\beta,\gamma;a)  \leq l_{n,2n}(\beta+1,\gamma;a).
 \end{split}
\end{equation}

From Remark \ref{nocommonzeros} we know that $ L_{n}^{(\beta,\gamma)}(x;a)$ and $ L_{n-1}^{(\beta+1,\gamma+1)}(x;a)$ do not have common zeros. But then, using \textit{(1)} of Theorem \ref{identityJl}, we conclude that $ L_n^{(\beta+1,\gamma)}(x;a)$ cannot have a common zero with $ L_{n-1}^{(\beta,\gamma)}(x;a)$ or $ L_n^{(\beta+1,\gamma+1)}(x;a)$. Thus, all the inequalities above are strict, which implies
\begin{equation*}
    L_n^{(\beta,\gamma)}(x;a)\prec  L_n^{(\beta+1,\gamma)}(x;a) \prec  xL_{n-1}^{(\beta+1,\gamma+1)}(x;a)    .
\end{equation*} 

 The proof of assertion \textit{(2)} is similar, as is for \textit{(3)}, although in this case strict inequalities follow from Remark \ref{nocommonzeros} and Theorem \ref{identityLH}.
\end{proof}

\section*{Acknowledgments}

The first author was partially supported by Simons Foundation Collaboration Grants for Mathematicians (grant 710499).
He also acknowledges the support of Junta de Andaluc\'{\i}a (research group FQM-229 and Instituto Interuniversitario Carlos I de F\'{\i}sica Te\'orica y Computacional) and of the project PID2021-124472NB-I00, funded by MCIN/AEI/10.13039/501100011033 and by ``ERDF A way of making Europe''. 

The authors thank the anonymous referee whose careful revision helped us eliminate many typos from the original version of this paper.


\end{document}